\documentclass[a4paper,11pt]{amsart}
\usepackage[pdfdisplaydoctitle=true,colorlinks=true,urlcolor=blue,citecolor=blue,linkcolor=blue,pdfstartview=FitH,pdfpagemode=None,bookmarksnumbered=true]{hyperref}
\usepackage{cmap} % make the PDF file "searchable and copyable"

\newtheorem{theorem}{Theorem}
\newtheorem{corollary}[theorem]{Corollary}

\newtheorem{proposition}[theorem]{Proposition}

\theoremstyle{definition}
\newtheorem{definition}[theorem]{Definition}

\theoremstyle{remark}

\theoremstyle{definition}
\newtheorem{example}[theorem]{Examples}

\newcommand{\Real}{\mathbb{R}}

\newcommand{\Circle}{\mathbb{S}^{1}}

\newcommand{\id}{\mathrm{id}}

\newcommand{\gstar}{\mathfrak{g}^{*}}

\newcommand{\DiffS}{\mathrm{Diff}^{\infty}(\mathbb{S}^{1})}

\newcommand{\VectS}{\mathrm{Vect}^{\infty}(\mathbb{S}^{1})}

\newcommand{\CS}{\mathrm{C}^{\infty}(\mathbb{S}^{1})}
\newcommand{\Vect}{\mathrm{Vect}}

\hypersetup{
    pdftitle={The periodic $\mathbf{b}\,$-equation and Euler equations on the cirlce}, %  Title
    pdfauthor={J. Escher and J. Seiler}, % Author
    pdfsubject={MSC 2010: 35Q53, 58D05}, % Subject
    pdfkeywords={Euler equation, diffeomorphisms group of the circle, Degasperis-Procesi equation}, % Keywords
    }

\begin{document}

%\title[Euler equations]{The periodic $\mathbf{b}\,$-equation and\\ Euler equations on the circle}%
\title[The periodic $\mathbf{b}\,$-equation and Euler equations on the circle]{The periodic $\mathbf{b}\,$-equation and\\ Euler equations on the circle}%

%    Information for first author
\author[J. Escher]{Joachim Escher} %
\address{Institute for Applied Mathematics, University of Hannover, D-30167 Hannover, Germany} %
\email{escher@ifam.uni-hannover.de} %

%    Information for second author
\author[J. Seiler]{J\"org Seiler} %
\address{Department of Mathematical Sciences,
Loughborough University,
Leicestershire LE11 3TU, United Kingdom
} %
\email{J.Seiler@lboro.ac.uk}%

\subjclass[2000]{35Q53, 58D05} % to be checked
\keywords{Euler equation, diffeomorphisms group of the circle, Degasperis-Procesi equation} %

%\date{}%
%\dedicatory{}%
%\commby{}%

% ----------------------------------------------------------------

\begin{abstract}
In this note we show that the periodic $b$-equation can only be realized as an
Euler equation on the Lie group $\DiffS$ of all smooth and orientiation preserving 
diffeomorphisms on the cirlce if $b=2$, i.e.\ for the Camassa-Holm equation. In this
case the inertia operator generating the metric on $\DiffS$ is given by $A=1-\partial_{x}^2$. 
In contrast, the Degasperis-Procesi equation, for which  $b=3$,  is not an Euler 
equation on $\DiffS$ for any inertia operator. 
Our result generalizes a recent result of B.\ Kolev \cite{Kol09}.
\end{abstract}

\maketitle

% ----------------------------------------------------------------

\section{Introduction}
\label{sec:intro}

In this note we are interested in the geometric interpretation of the so-called $b$\,-equation 
\begin{equation}\label{eq:b-equation}
    m_{t} = - (m_{x}u + bmu_{x}),\qquad t\in\Real,\ x\in\mathbb{S}^1,
\end{equation}
with the momentum variable $m$ given by
\begin{equation*}
    m = u - u_{xx},
\end{equation*}
and where $b$ stands for a real parameter,
cf.\  \cite{DP99, DHH02, HW03}. It was shown in \cite{DP99, HW03, MN02, Iva05} that equation \eqref{eq:b-equation}  is   
asymptotically integrable, a necessary condition for complete integrability, but only for the values $b=2$ and $b=3$.
In case $b=2$ we recover the Camassa-Holm equation (CH)
\begin{equation}\label{eq:CH}
    u_{t} - u_{txx} + 3uu_{x} - 2u_{x}u_{xx} - uu_{xxx} = 0,
\end{equation}
while for $b=3$ we obtain the Degasperis-Procesi equation (DP)
%\begin{equation}\label{eq:DP}
\begin{equation*}
    u_{t} - u_{txx} + 4uu_{x} - 3u_{x}u_{xx} - uu_{xxx} = 0.
\end{equation*}

Independent of (asymptotical) integrability, equation \eqref{eq:b-equation}  possesses some
hydrodynamic relevance, as described for instance in~\cite{Joh03,Iva07, CL09}. Each of these equations models the unidirectional  
irrotational free surface flow of a shallow layer of an inviscid fluid moving under the influence of gravity over a flat bed. 
In these models, $u(t,x)$ represents the wave's height at the moment $t$ and at position $x$ above the flat bottom.

The periodic Camassa-Holm equation is known to correspond to the geodesic flow  with respect to the metric induced by the 
inertia operator $1-\partial_x^2$ on the diffeomorphism group of the circle, cf.\  \cite{Kou99}. Local existence of the geodesics 
and properties of the Riemannian exponential map were studied in \cite{CK02, CK03}. 
The whole family of $b\,$-equations and in particular (DP) can be realized as (in general) non-metric Euler equations, 
i.e.\ as geodesic flows with respect to a linear connection which is not necessarily Riemannian in the sense that there 
may not exist a Riemannian metric which is preserved by this connection, cf.\  \cite{EK10}. 

Besides various common properites of the individual members of the $b$\,-equation, there are also significant differences to report on. 
It is known that solutions of the CH-equation  preserve the $H^1$-norm in time and that (CH) 
possesses global in time weak solutions, cf.\ \cite{CE98, CE98c, BC07}. In particular there are no shock waves for (CH), although finite 
time blow of classical solutions occurs, but in form of wave breaking: solutions to (CH) stay continuous and bounded but their slopes may 
blow up in finite time, cf.\  \cite{CE98b, CE00}. Wave breaking is also observed for the (DP) but in a weaker form. It seems that the 
$H^1$-norm of solutions of (DP) cannot be uniformly bounded, but $L_\infty$-bounds for large classes 
of initital values are available \cite{ElY06, ELY06b, EY08}. Moreover shock waves, i.e.\ discontinuous global travelling wave soltions 
are known to exist. Indeed it was shown in \cite{E07} that
\begin{equation*}
    u_c(t,x)=\frac{\sinh\left(x-[x]-{1/2}\right)}{t\cosh\left({1/2}\right)+c\sinh\left({1/2}\right)}, \quad x\in\Real/\mathbb{Z},
\end{equation*}
is for any $c>0$ a global weak solution to the (DP) equation.

In this note we disclose a further difference between the (CH) and the (DP) equation, by proving that in a fairly large class of 
Riemannian metrics on $\DiffS$ it is impossible to realize (DP) as a geodesic flow.
\medskip

The note is organized as follows. In Section~\ref{sec:general case}, we first introduce the concept geodesic flows and Euler equations 
on a general Lie group. Subsequently, in Section~\ref{sec:diff circle},  the important special case of $\DiffS$ is discussed and 
Section~\ref{sec: b-eq} contains the proof of our main result. 

% ----------------------------------------------------------------

\section{The Euler equation on a general Lie group}
\label{sec:general case}
\noindent
In his famous article \cite{Arn66} Arnold established a deep geometrical connection between the Euler equations for a perfect fluid in two and three 
dimensions and the geodesic flow for right-invariant metrics on the Lie group of volume-preserving diffeomorphisms. 
After Arnold's fundamental work a lot of effort was  devoted to understand the geometric structure of other physical systems with a Lie group as configuration space. 

The general Euler equation was derived initially for the Levi-Civita connection of a one-sided invariant Riemannian metric on a Lie group $G$ 
(see \cite{Arn66} or \cite{AK98}) but the theory is even valid in the more general setting of a one-sided invariant linear connection, see \cite{EK10}.

A right invariant metric on a Lie group $G$ is determined by its value at the unit element $e$ of the group, i.e.\ by an inner product on its Lie algebra $\mathfrak{g}$. 
This inner product can be expressed in terms of  a symmetric linear operator $A\,: \mathfrak{g}\to \gstar$, i.e.
\begin{equation*}
 \langle Au, v\rangle=\langle Av, u\rangle,\qquad{\hbox{for all}}\ u,\, v\in \mathfrak{g},
\end{equation*}
where $\gstar$ is the dual space of $\mathfrak{g}$ and $\langle\cdot,\cdot\rangle$ denotes the duality pairing on $\gstar\times\mathfrak{g}$. 
Each symmetric isomorphism  $A\,: \mathfrak{g}\to \gstar$ is called an {\em{inertia operator}} on $G$. The corresponding metric on $G$ induced by $A$ is denoted by $\rho_A$. 

Let $\nabla$ denote the Levi-Civita connection on $G$ induced by $\rho_A$. Then

\begin{equation}\label{eq:connection}
    \nabla_{\xi_{u}} \, \xi_{v} = \frac{1}{2} [\xi_{u},\xi_{v}] + B(\xi_{u},\xi_{v}),
\end{equation}
where $\xi_u$ is the right invariant vector field on $G$, generated by $u\in\mathfrak{g}$. Moreover, $[\cdot,\cdot]$ is the Lie bracket on $\Vect(G)$, 
the smooth sections of the tangent bundle over $G$, and the bilinear operator $B$ is called {\em{Christoffel operator}}. It is defined by the following formula
\begin{equation}\label{christ}
    B(u,v) = \frac{1}{2}\Big[ (\mathrm{ad}_{u})^{*}(v) + (\mathrm{ad}_{v})^{*}(u)\Big],
\end{equation}
where $(\mathrm{ad}_{u})^{*}$ is the adjoint with respect to $\rho_A$ of the natural action of $\mathfrak{g}$ on itself, given by
\begin{equation}\label{ad}
 \mathrm{ad}_{u}\,:\, \mathfrak{g}\to \mathfrak{g}, \quad v\mapsto [u,v].
\end{equation}
A proof of the above statments as well as of the following proposition can be found in \cite{EK10}.

\begin{proposition}
A smooth curve $g(t)$ on a Lie group $G$ is a \emph{geodesic} for a right invariant linear connection $\nabla$ defined 
by~\eqref{eq:connection} iff its Eulerian velocity $u = {g}'\circ g^{-1}$ satisfies the first order equation
\begin{equation}\label{eq:euler}
    u_{t} = - B(u,u).
\end{equation}
Equation \eqref{eq:euler} is known as the \emph{Euler equation}.
\end{proposition}

% ----------------------------------------------------------------

\section{The Euler equation on $\DiffS$}
\label{sec:diff circle}
\noindent
Since the tangent bundle $T\Circle \simeq \Circle \times \Real$ is trivial, $\VectS$, the space of smooth vector fields on $\Circle$, 
can be identified with $\CS$, the space of real smooth functions on $\Circle$. Furthermore, the group $\DiffS$ is naturally equipped 
with a Fr\'{e}chet manifold structure modeled over $\CS$, cf.\ \cite{EK10}.  
The Lie bracket on $\VectS \simeq \CS$ is given by\footnote{Notice that this bracket differs from the usual bracket of vector fields by a sign.}
\begin{equation*}
    [u,v] = u_{x}v - uv_{x} .
\end{equation*}

The topological dual space of $\VectS\simeq \CS$ is given by the distributions $\Vect'(\mathbb{S}^1)$ on $\mathbb{S}^1$. 
In order to get a convenient representation of the Christoffel operator $B$ we restrict ourselves to $\Vect^\ast(\mathbb{S}^1)$, 
the set of all regular distributions which may be represented by smooth densities, i.e.\ $T\in\Vect^\ast(\mathbb{S}^1)$ iff
there is a $\varrho\in\CS$ such that 
\begin{equation*}
 \langle T,\varphi\rangle=\int_{\mathbb{S}^1}\varrho\varphi\,dx\quad\hbox{for all}\quad  \varphi\in\CS.
\end{equation*} 

By Riesz' representation theorem we may identify
$\Vect^\ast(\mathbb{S}^1)\simeq \CS$. In the following we denote by $\mathcal{L}_{is}^{sym}(\CS)$ the set of all continuous isomorphisms 
on $\CS$, which are symmetric with respect to the $L_2(\mathbb{S}^1)$ inner product.

\begin{definition} 
Each $A\in\mathcal{L}_{is}^{sym}(\CS)$ is called a {\em{regular inertia operator}} on $\DiffS$.
\end{definition}

\begin{proposition}
Given $A\in\mathcal{L}_{is}^{sym}(\CS)$, we have that
\begin{equation*}
 B(u,v)=\frac{1}{2}A^{-1}[2Au\cdot v_x + 2Av\cdot u_x+ u\cdot (Av)_x+ v\cdot(Au)_x]
\end{equation*}
for all $u,\;v\in \CS$.
\end{proposition}

\begin{proof}
Let $u,\,v,\,w\in\CS$ be given. Recalling (\ref{ad}), integration by parts yields
\begin{align*}
    \rho_A(\mathrm{(ad}_{u})^\ast v,w) & = \rho_A(v,\mathrm{ad}_{u}w)= \int_{\mathbb{S}^1}A v\cdot(u_xw-uw_x)\, dx \\
		& = \int_{\mathbb{S}^1}\left[(Av)u_x+((Av)\cdot u)_x\right] w\,dx.
\end{align*}
This shows that 
\begin{equation*} 
 \mathrm{(ad}_{u})^\ast v = 2(Av) u_x + u(Av)_x. 
\end{equation*}
Symmetrization of this formula completes the proof, cf.\  \eqref{christ}.
\end{proof}

\begin{example}\label{exs} 
{\rm{It may be instructive to discuss two paradigmatic examples. 
  \begin{enumerate}
   \item First we choose $A=\id$. Then $B(u,u)=-3uu_x$ and the corresponding Euler equation $u_t+3uu_x=0$ 
    is known as the periodic inviscid Burgers equation, see e.g.\  \cite{Bur48,Kat75}.
   \item Next we choose $A=\id -\partial_x^2$. Then the Euler equation reads as 
    $u_t=-(1-\partial_x^2)^{-1}\left(3uu_{x} - 2u_{x}u_{xx} - uu_{xxx}\right)$, which equivalent to the 
    periodic Camassa-Holm equation, cf.\  \eqref{eq:CH}.
 \end{enumerate}
}}
\end{example}

% ----------------------------------------------------------------

\section{The family of $b$-equations}\label{sec: b-eq}

Each $A\in\mathcal{L}_{is}^{sym}(\CS)$ induces an Euler equation on $\DiffS$. Conversely, 
given $b\in\Real$, we may ask whether there exists a regular inertia operator such that 
the $b$\,-equation is the corresponding Euler equation on $\DiffS$. We know from 
Example \ref{exs}.2 that the answer is positive if $b=2$. The following result shows that 
the answer is negative if $b\ne 2$.

\begin{theorem}\label{main} 
Let $b\in\Real$ be given and suppose that there is a regular inertia operator $A\in\mathcal{L}_{is}^{sym}(\CS)$ such that the $b\,$-equation
\begin{equation*}
 m_t=-(m_x u+bm u_x),\qquad m=u-u_{xx}
\end{equation*}
is the Euler equation on $\DiffS$ with respect to $\rho_A$. Then  $b=2$ and $A=\mathrm{id}-\partial_{x}^2$. 
\end{theorem}

\begin{corollary}
The Degasperis-Procesi equation 
\begin{equation*}
 m_t=-(m_x u+3m u_x),\qquad m=u-u_{xx}
\end{equation*}
cannot be realized as an Euler equation for any regular inertia operator $A\in\mathcal{L}_{is}^{sym}(\CS)$.
\end{corollary}

\begin{proof}[Proof of Theorem~{\rm\ref{main}}] 
Let $b\in\mathbb{R}$ be given and assume that the $b\,$-equation is the Euler equation on $\DiffS$ with respect to $\rho_A$. 
Letting $L=1-\partial_x^2$,  we then get
\begin{equation}\label{b-euler}
 A^{-1}\Big(2(Au)u_x+u(Au)_x\Big)=L^{-1}\Big(b(Lu)u_x+u(Lu)_x\Big)
 \end{equation}
for all $u\in\CS$. 

\medskip

(a) Let $\mathbf{1}$ denote the constant function with value $1$. Choosing $u=\mathbf{1}$ in \eqref{b-euler}, we get
$A^{-1}(\mathbf{1}(A\mathbf{1})_x)=0$ and hence $(A\mathbf{1})_x=0$, i.e.\ $A\mathbf{1}$ is constant. Scaling \eqref{b-euler}, 
we may assume that $A\mathbf{1}=\mathbf{1}$. Next we replace $u$ by  $u+\lambda$ in \eqref{b-euler}. 
Then we find for the left-hand side that
\begin{align*}
 \frac{1}{\lambda}A^{-1}&
 \Big(2\big(A(u+\lambda)\big)(u+\lambda)_x+(u+\lambda)\big(A(u+\lambda)\big)_x\Big)\\
 =&\frac{1}{\lambda}A^{-1}\Big(2\big((Au)+\lambda\big)u_x+(u+\lambda)(Au)_x\Big)\\
 =&A^{-1}\Big(\frac{2(Au)u_x+u(Au)_x}{\lambda}+2u_x+(Au)_x\Big)\\
 &\xrightarrow{\lambda\to\infty}A^{-1}\big(2u_x+(Au)_x\big),
\end{align*} 
and similarly for the right-hand side:
\begin{align*}
 \frac{1}{\lambda}L^{-1}&
 \Big(b\big(L(u+\lambda)\big)(u+\lambda)_x+(u+\lambda)\big(L(u+\lambda)\big)_x\Big)\\
 &\xrightarrow{\lambda\to\infty}L^{-1}\big(bu_x+(Lu)_x\big).
\end{align*} 
Combing these limits, we conclude that
\begin{equation}\label{eq-lim}
 A^{-1}\big(2u_x+(Au)_x\big)=L^{-1}\big(bu_x+(Lu)_x\big)
\end{equation} 
for all $u\in\CS$. Setting $u_n=e^{inx}$, we find that 
\begin{equation*}
 L^{-1}\big(b(u_n)_x+(Lu_n)_x\big)=i\alpha_n u_n,
\end{equation*}
where $\alpha_n:=n+\left(b{n}/(1+n^2)\right)$. Applying $A$ to \eqref{eq-lim} with $u=u_n$ 
thus yields 
\begin{equation*}
 2inu_n+(Au_n)_x=i\alpha_nAu_n. 
\end{equation*}
Therefore $v_n:=Au_n$ solves the ordinary differential equation 
\begin{equation}\label{ODE}
 v^\prime-i\alpha_nv=-2in u_n.
\end{equation}
For $n\ne0$, let us solve \eqref{ODE} explicitly. Assume first that $b=0$. Then 
 $$v(x)=(c-2inx)u_n$$%e^{inx}$$ 
for some constant $c$. But this function is never $2\pi$-periodic. Thus we must have $b\ne 0$. However, in this case
\begin{equation}\label{solu}
 v_n=Au_n=\gamma_ne^{i\alpha_nx}+\beta_n u_n\qquad (n\ne0)%e^{inx}   
\end{equation} 
with $\beta_n=\frac{2(1+n^2)}{b}$ and suitable constants $\gamma_n$. 

\medskip
(b) Assume that all $\gamma_n$ vanish, i.e.\ $Au_n=\beta_n u_n$ for all $n\ne0$ and $A\mathbf{1}=\mathbf{1}$. 
In particular, $A$ is a Fourier multiplication operator and thus commutes with $L$. Therefore we can write 
\eqref{b-euler} as 
\begin{equation*}
 L\Big(2(Au)u_x+u(Au)_x\Big)=A\Big(b(Lu)u_x+u(Lu)_x\Big).
\end{equation*}
Inserting $u=u_n$ a direct computation yields
\begin{equation*}
 3(1+4n^2)\beta_n=(1+b)(1+n^2)\beta_{2n}. 
\end{equation*}
Using that $\beta_n=2(1+n^2)/b$, this is equivalent to $b=2$. Then $\beta_n=1+n^2$ and therefore 
$A$ coincides with $L$.  

\medskip

(c) Let us assume there is a $p\in\mathbb{Z}\setminus\{0\}$ such that $\gamma_{p}\ne 0$. 
We shall derive a contradiction. Since $v_p=Au_p$ must be $2\pi$-periodic, $\alpha_p$ is an integer. 
This implies that $b=k(1+p^2)/p$ for some non-zero integer $k$. We set
\begin{equation}\label{m=an}
 m:=\alpha_{p}.
\end{equation}
Observe that $m\ne p$, since $b\ne 0$. Thus $(u_p\vert u_m)_{L_2}=0$ and \eqref{solu} implies that
\begin{equation*}
 (Au_p\vert\, u_m)_{L_2}=(\gamma_p\, e^{imx}\vert\, u_m)_{L_2}=\gamma_p.
\end{equation*}
By the symmetry of $A$ we also find that
\begin{equation*}
 \gamma_p=(A u_p\vert\, u_m)_{L_2}=(u_p\vert\, A u_m)_{L_2}=\overline{\gamma_m}(u_p\vert\, e^{i\alpha_mx})_{L_2}.
\end{equation*}
Since $\gamma_p\ne0$ we must have $\gamma_m\ne0$. Again by periodicity we conclude that $\alpha_m\in \mathbb{Z}$. 
But then $\alpha_m=p$, since otherwise we would have $(u_p\vert e^{i\alpha_m x})_{L_2}=0$ and thus again $\gamma_p=0$. 
We know already that $b=k(1+p^2)/p$. 
Thus $m=\alpha_p=p+k$ by \eqref{m=an} and the definition $\alpha_p$. Now we calculate
\begin{align*}
 p&=\alpha_m=\alpha_{p+k}=p+k+b\frac{p+k}{1+(p+k)^2}\\
  &= p+k+\frac{k(1+p^2)}{p}\cdot\frac{p+k}{1+(p+k)^2},
\end{align*}
and we find
\begin{equation*}
 p\left(1+(p+k)^2\right)k+k(1+p^2)(p+k)=0.
\end{equation*}
Observing that $k\ne 0$, an elementary calculation yields
\begin{equation}\label{rel}
 2 p^3 + 3 p^2k+pk+2p+k=0.
\end{equation}
From this we conclude that there is an $l\in\mathbb{Z}$ such that $k=p\,l$. With this we infer from \eqref{rel}  that
\begin{equation*}
 (l+2)\left((l+1)p^2+1\right)=0.
\end{equation*}
The only integer solution of this equation is $l=-2$. In fact, the solution $l=-\frac{1}{p^2}-1$ is only possible if 
$p^2=1$ and thus again $l=-2$, since for $p^2\ne 1$ we have $l\not\in\mathbb{Z}$. Therefore $b=-2(1+p^2)$ and thus $\alpha_p=-p$. 

Moreover, we can conclude that $\gamma_n=0$ whenever $n\ne0$ does not coincide with $p$ or $-p$, since otherwise the same calculation 
as before would show $b=-2(1+n^2)$ contradicting $b=-2(1+p^2)$. 

Now insert $u=u_p$ in \eqref{b-euler}. The left-hand side 
then equals 
 $$A^{-1}\Big(ip\gamma_p\mathbf{1}-3ipu_{2p}\Big)=ip\gamma_p\mathbf{1}-\frac{3ip}{\beta_{2p}}u_{2p};$$
for the latter identity note that $2p$ does not coincide with $p$ or $-p$, so that $\gamma_{2p}=0$, and hence $A^{-1}u_{2p}=u_{2p}/\beta_{2p}$. 
Note also that $\beta_p=-1$. For the right-hand side we get 
 $$i(1+b)\frac{p(1+p^2)}{1+4p^2}u_{2p}.$$
Comparing both expressions we conclude that $p\gamma_p=0$ which is a contradiction to $p\ne0$ and $\gamma_p\ne0$. 
\end{proof} 

\bibliographystyle{amsplain}

\bibliography{JEJS}

\end{document}